\documentclass[12pt]{article}

\usepackage{tikz}
\usepackage{subfigure}
\usepackage[english]{babel}

\usepackage[center]{caption2}
\usepackage{amsfonts,amssymb,amsmath,latexsym,amsthm}
\usepackage{multirow}

\newtheorem{thm}{Theorem}
\newtheorem{cor}[thm]{Corollary}
\newtheorem{lem}[thm]{Lemma}
\newtheorem{prop}{Proposition}
\newtheorem{conj}[thm]{Conjecture}

\newtheorem{claim}{Claim}

\begin{document}

\title{Odd induced subgraphs in graphs with treewidth at most two
\thanks{The work was supported by NNSF of China (No. 11671376) and  NSF of Anhui Province (No. 1708085MA18) and  the Fundamental Research Funds for the Central Universities.}
}
\author{Xinmin Hou$^a$, \quad Lei Yu$^b$, \quad Jiaao Li$^c$, \quad Boyuan Liu$^d$\\
\small $^{a,b,d}$ Key Laboratory of Wu Wen-Tsun Mathematics\\
\small School of Mathematical Sciences\\
\small University of Science and Technology of China\\
\small Hefei, Anhui 230026, China.\\
\small $^{c}$Department of Mathematics\\
\small West Virginia University\\
\small Morgantown, WV 26506, USA
}
\date{}
\maketitle

\begin{abstract}
A long-standing conjecture asserts that there exists a constant  $c>0$ such that every graph of order $n$ without isolated vertices contains an induced subgraph of order at least $cn$ with all degrees odd. Scott (1992) proved that every graph $G$ has an induced subgraph of order at least $|V(G)|/(2\chi(G))$ with all degrees odd, where $\chi(G)$ is the chromatic number of $G$,  this implies the conjecture for graphs with { bounded} chromatic number. But the factor $1/(2\chi(G))$ seems to be not best possible, for example, Radcliffe and Scott (1995) proved $c=\frac 23$ for trees, Berman, Wang and Wargo (1997) showed that $c=\frac 25$  for graphs with maximum degree $3$, so it is interesting to determine the exact value of $c$ for special family of graphs. In this paper, we further confirm the conjecture for graphs with treewidth at most 2 with $c=\frac{2}{5}$, and the bound is best possible.
\end{abstract}

\section{Introduction}

Gallai~\cite{Lovasz-CPE1979} proved that for every graph $G$, the vertex set $V(G)$ can be partitioned into two sets, each of which induces a subgraph with all degrees even.
This implies that every graph of order $n$ contains an induced subgraph of order at least $\lceil \frac{n}{2}\rceil$ with all degrees even, and this  is best possible by considering paths. This motivates us to consider the problem that how large we can find an induced subgraph with all degrees odd.
We call a graph with all degrees odd an {\it odd graph}. Let $f(G)$ denote the maximum order of an odd induced subgraph in a graph $G$. The following long-standing conjecture was cited by Caro in~\cite{Caro-DM1994} as ``part of the graph theory folklore'' and the origin is unclear.

\begin{conj}~\label{CONJ: OSC}
There exists a constant $c>0$ such that for every graph $G$ without isolated vertices, $f(G)\geq c|V(G)|$.
\end{conj}


The ``without isolated vertices'' constraint is natural because an odd graph does not contain isolated vertices.
Many results related to Conjecture~\ref{CONJ: OSC} have been obtained in literatures. In particular, Caro~\cite{Caro-DM1994} proved that $f(G)\geq (1-o(1))\sqrt{|V(G)|/6}$, laterly, Scott~\cite{Scott-CPC92} improved the lower bound to $\frac{c|V(G)|}{\log{|V(G)|}}$ for some $c>0$, in the same paper, Scott also proved that every graph $G$ has an odd induced subgraph of order at least $|V(G)|/(2\chi(G))$, where $\chi(G)$ is the chromatic number of $G$, this implies the conjecture for graphs with { bounded} chromatic number. But the factor $1/(2\chi(G))$ seems to be not best possible, for example,  Radcliffe and Scott~\cite{RS-DM95} confirmed the conjecture for trees (graphs with treewidth one) with $c=\frac{2}{3}$ and Berman, Wang and Wargo~\cite{BWW-AJC97} proved the conjecture for graphs with maximum degree $3$ with $c=\frac{2}{5}$. In this paper, we further confirm Conjecture~\ref{CONJ: OSC} for graphs with treewidth at most 2 with $c=\frac 25$, and the value of $c$ is best possible.


 A {\it tree decomposition} of a graph $G$ is a tree $T$, where

 (1) Each vertex $i$ of $T$ is labeled by a subset $B_i$ of vertices of $G$.

 (2) Each edge of $G$ is in a subgraph induced by at least one of the $B_i$,

 (3) For every three vertices $i,j,k$ in $T$ with $j$ lying on the path from $i$ to $k$ in $T$, $B_i\cap B_k\subseteq B_j$.\\
\noindent The {\it tree-width} tw($G$) of $G$ is the minimum integer $p$ such that there exists a tree decomposition of $G$ with all subsets of cardinality at most $p+1$. {Tree-decomposition is one of the most general and effective techniques for
designing efficient algorithms, and a tree-like structure allows us to solve certain difficult problems. It is well-known that a connected graph has treewidth one if and only if it is tree. In terms of treewidth, the result of Radcliffe and Scott~\cite{RS-DM95} can be restated as follows.
\begin{thm}\label{THM: RSMain}\cite{RS-DM95}
 For any connected graph $T$ with $tw(T)=1$,  $f(T)\geq 2\lfloor\frac{|V(T)|+1}{3}\rfloor$ .
\end{thm}
}

The following theorem is our main result.
\begin{thm}\label{THM: Main}
For every graph $G$ with $tw(G)\le 2$ and without isolated vertices, $f(G)\geq \frac{2}{5}|V(G)|$.
\end{thm}

The lower bound is sharp by considering the graph of which each component is a cycle of length $5$. { We remark that graph with treewidth at most two is also known as {\em $K_4$-minor-free graph}, see Proposition \ref{PROP: p1} in section 3. Some upper and lower bounds on graphs with small treewidth are also discussed in the last section.}

In this paper, standard notation follows from~\cite{Diestel-GT00}. In particular, for a graph $G$ and a set $S\subseteq V(G)$, let $G[S]$ be the subgraph induced by $S$ and let $N_G(S)$ be the union of neighbors of vertices in $S$,  for a vertex $u\in V(G)$, let $N_G^1(u)=\{x\ |\ x\in N_G(u) \mbox{ and } d_G(x)=1\}$ and $N_G^2(u)=\{x\ |\ x\in N_G(u) \mbox{ and } d_G(x)=2\}$, and denote $N_G^2(u,v)=N_G^2(u)\cap N_G^2(v)$.
A vertex of degree $k$ is called a {\it $k$-vertex}.  Define $S_G(u)=\{x\ |\ x\in N_G(u)$ with $d_G(x)\geq 3$  or there exists a vertex $z\in N_G^2(u,x)\}$. Let $D_G(u)=|S_G(u)|$. For two sets $S, T$, we use $S\setminus T$ denote $S-(S\cap T)$.

The rest of the paper is arranged as follows. In section 2, we establish  structural properties of minimum counterexample of Theorem~\ref{THM: Main}. Then the proof of Theorem~\ref{THM: Main} is presented in Section 3, and in the last section, we give some discussions.

\section{Properties of minimal counterexample}

Let $G$ be a minimum counterexample of Theorem~\ref{THM: Main} with respect to the order of $G$. The main idea of the proof is as the following.
We first pick some set $V_0\subset V(G)$  so that $G'=G-V_0$ has no isolated vertex,  by the minimality of $G$, $G'$ has an odd induced subgraph $H'$ with $|V(H')|/|V(G')|\ge 2/5$. We will find a set $S_0\subset V_0$ with $|S_0|\geq \frac{2}{5}|V_0|$ such that $S_0\cup V(H')$ induces an odd induced subgraph $H$ of $G$. We should be careful to remain the parity of the degrees of the vertices in $N_G(S_0)\cap V(H')$ and $S_0\cap N_G(V(H'))$. Here we allow $V(G')=\emptyset$.

\begin{lem}\label{LEM: P1}
Let $u$ be a vertex of $G$ with $D_G(u)=1$ and let $S_G(u)=\{v\}$.
Then $N_G^1(u)\cup N_G^2(u,v)=\emptyset$.
\end{lem}
\begin{proof}[Proof]
Suppose to the contrary that $G$ has a vertex $u$ with $D_G(u)=1$ and $N_G^1(u)\cup N_G^2(u,v)\not=\emptyset$. Let $t_1=|N_G^1(u)|$ and $t_2=|N_G^2(u,v)|$. Then $t_1+t_2>0$.

\noindent{\bf Case 1}. $|N_G^1(v)|\le 1$.

Set $V_0=N_G^1(u)\cup N_G^2(u,v)\cup \{u,v\}\cup N_G^1(v)$ and $G'=G-V_0$. Then $G'$ has no isolated vertex, so, by the minimality of $G$, $G'$ has an odd induced subgraph $H'$ with $|V(H')|\ge \frac 25 |V(G')|$. Now let $S_0=V_0\setminus (N_G^1(v)\cup\{v\})$.  Then $G[S_0]\cong K_{1, t_1+t_2}$, and so  $G[S_0]$ contains an odd induced subgraph $K=K_{1,t}$ with $t=t_1+t_2$ if $t_1+t_2$ is odd or $t=t_1+t_2-1$ if $t_1+t_2$ is even. So $(t+1)/|V_0|\ge (t+1)/(t_1+t_2+3)\ge 2/5$. Furthermore, we have $N_G(V(K))\cap V(H')=\emptyset$ and $V(K)\cap N_G(V(H'))=\emptyset$. Hence $H=H'\cup K$ is an odd induced subgraph of $G$ with $|V(H)|\ge \frac 25|V(G)|$, a contradiction.





\noindent{\bf Case 2}. $|N_G^1(v)|\ge 2$.

Choose a vertex $x\in N_G^1(v)$ and set $V_0=N_G^1(u)\cup N_G^2(u,v)\cup \{u,x\}$ and $G'=G- V_0$. Then $G'$ has no isolated vertex and so, by the minimality of $G$, $G'$ has an odd induced subgraph $H'$ with $|V(H')|/|V(G')|\ge 2/5$.

\begin{claim}\label{CLAIM: l5c1}
$v$ must be in $V(H')$.
\end{claim}
Suppose to the contrary that $v\notin V(H')$. Set $S_0=V_0\setminus\{x\}$, then $G[S_0]\cong K_{1, t_1+t_2}$, and so $G[S_0]$ contains an odd induced subgraph $K=K_{1, t}$ with $t=t_1+t_2$ or $t=t_1+t_2-1$ with respect to the parity of $t_1+t_2$.  Note that $(t+1)/|V_0|=(t+1)/(t_1+t_2+2)> 2/5$, $N_G(V(K))\cap V(H')=\emptyset$ and $V(K)\cap N_G(V(H'))=\emptyset$. Therefore, $H=K\cup H'$ is an odd induced subgraph of $G$ with $|V(H)|/|V(G)|>2/5$, a contradiction. The claim is true.

\vspace{5pt}
Now suppose $v\in V(H')$.

\begin{claim}\label{l5c2}
We have $t_2\le t_1$.
\end{claim}
Suppose to the contrary that $t_2\geq t_1+1$. Set $S_0=N_G^2(u,v)\cup \{x\}$ if $t_2$ is odd or $S_0=N_G^2(u,v)$ if $t_2$ is even, then $S_0\cup V(H')$ still induces an odd subgraph $H$ of $G$ with $|V(H)|=|S_0|+|V(H')|\ge \frac  25|V_0|+\frac 25|V(G')|=\frac 25|V(G)|$,  a contradiction, where the second inequality holds since $|S_0|/|V_0|\ge |S_0|/(t_1+t_2+2)\ge |S_0|/(2t_2+1)\ge 2/5$. Hence the claim holds.

Now suppose $t_2\leq t_1$ and let $T_1$ (resp. $T_2$) be a maximum subset of odd (resp. even) order in $N_G^1(u)$. Set $S_0=T_1\cup\{u\}$ if $uv\notin E(G)$ or $S_0=T_2\cup \{u,x\}$ if $uv\in E(G)$. In both cases, $|S_0|/|V_0|=|S_0|/(t_1+t_2+2)\ge |S_0|/(2t_1+2)\geq 2/5$ unless $t_1=t_2=2$ and $uv\notin E(G)$.
Then $S_0\cup V(H')$ induces an odd subgraph $H$ of $G$ with $|V(H)|=|S_0|+|V(H')|\ge 2/5|V_0|+2/5|V(G')|=2/5|V(G)|$, a contradiction. For $t_1=t_2=2$ and $uv\notin E(G)$, reset $V_0=N_G^1(u)\cup N_G^2(u,v)\cup \{u\}=N_G(u)\cup\{u\}$ and let $G'=G-V_0$, then $G'$ has no isolated vertex and so, by the minimality of $G$, $G'$ has an odd induced subgraph $H'$ with $|V(H')|/|V(G')|\ge 2/5$.  Let $N_G^1(u)=\{a,b\}$ and set $S_0=\{a,u\}$.
Then $S_0\cup V(H')$ induces an odd subgraph $H$ of $G$ with $|V(H)|\ge 2+\frac 25|V(G')|=\frac 25|V(G)|$, a contradiction again.

This completes the proof of the lemma.
\end{proof}

\begin{lem}\label{LEM: P2}
Let $u$ be a vertex of $G$ with $D_G(u)=2$ and let $S_G(u)=\{v,w\}$. Then $N_G^1(u)\cup N_G^2(u,v)\cup N_G^2(u,w)=\emptyset$.
\end{lem}
\begin{proof}[Proof]
Suppose to the contrary that $G$ has a vertex $u$ with $S_G(u)=\{v,w\}$ and $N_G^1(u)\cup N_G^2(u,v)\cup N_G^2(u,w)\not=\emptyset$. Let $t_1=|N_G^1(u)|$, $t_2=|N_G^2(u,v)|$ and $t_3=|N_G^2(u,w)|$. Then $t_1+t_2+t_3>0$. Let $\bar{N}_G^2(v,w)=N_G^2(v, w)\setminus\{u\}$.


\vspace{5pt}
\begin{claim}\label{CLAIM: l6c2}
If $N^1_G(v)=\emptyset$ then $N^1_G(w)\cup \bar{N}_G^2(v,w)\not=\emptyset$; symmetrically, if $N^1_G(w)=\emptyset$ then $N^1_G(v)\cup \bar{N}_G^2(v,w)\not=\emptyset$.
\end{claim}
We only prove the first statement, the second one can be proved similarly. Suppose to the contrary that $N^1_G(w)\cup \bar{N}_G^2(v,w)=\emptyset$. Set $V_0=N_G(u)\cup \{u,v,w\}$ and $G'=G-V_0$. Then $G'$ has no isolated vertex and so, by the minimality of $G$, $G'$ has an odd induced subgraph $H'$ with $|V(H')|\ge \frac 25 |V(G')|$. Let $S$ be a maximum subset of $N^1_G(u)\cup N^2_G(u,v)\cup N_G^2(u,w)$ so that $s=|S|$ is odd. Then $S_0=S\cup\{u\}$ induces an odd subgraph $K\cong K_{1,s}$ of $G[V_0]$, furthermore $|S_0|/|V_0|\ge (s+1)/(t_1+t_2+t_3+3)\ge 2/5$.
Note that $N_G(S_0)\cap V(H')=\emptyset$ and $S_0\cap N_G(V(H'))=\emptyset$. Therefore, $H=K\cup H'$ is an odd induced subgraph of  $G$ with $|V(H)|\ge \frac 25|V_0|+\frac 25|V(G')|=\frac 25|V(G)|$, a contradiction. The claim is true.

\vspace{5pt}
\noindent{\bf Case 1}. $N_G^1(v)=\emptyset$.

\vspace{5pt}

\noindent{\bf Subcase 1.1}. $|N_G(w)\setminus (N_G^2(u,w)\cup \{u,v\})|\le 1$.

Note that $|N_G(w)\setminus (N_G^2(u,w)\cup \{u,v\})|\le 1$ implies that $|N_G^1(w)\cup \bar{N}_G^2(v,w)|\le 1$. By Claim~\ref{CLAIM: l6c2}, $|N_G^1(w)\cup \bar{N}_G^2(v,w)|= 1$ and so $N_G(w)\setminus (N_G^2(u,w)\cup \{u,v\})=N_G^1(w)\cup \bar{N}_G^2(v,w)$. Let $N_G^1(w)\cup \bar{N}_G^2(v,w)=\{x\}$  and set $V_0=N_G(u)\cup \{u,v,w,x\}$ and $G'=G- V_0$. Then $G'$ has no isolated vertex and so, by the minimality of $G$, $G'$ has an odd induced subgraph $H'$ with $|V(H')|\ge \frac 25 |V(G')|$. Let $S$ be a maximum subset of $N^1_G(u)\cup N^2_G(u,v)\cup N^2_G(u,w)$ so that $s=|S|$ is odd. Then $S_0=S\cup\{u\}$ induces an odd subgraph $K\cong K_{1,s}$ of $G[V_0]$, furthermore $|S_0|/|V_0|\ge (s+1)/(t_1+t_2+t_3+4)\ge 2/5$ unless $t_1+t_2+t_3=2$.
Note that $N_G(S_0)\cap V(H')=\emptyset$ and $S_0\cap N_G(V(H'))=\emptyset$. Hence $H=K\cup H'$ is an odd induced subgraph of  $G$ with $|V(H)|\ge \frac 25|V_0|+\frac 25|V(G')|=\frac 25|V(G)|$ provided that $t_1+t_2+t_3\not=2$, a contradiction.

For $t_1+t_2+t_3=2$, notice that $E_G(w, V(G'))=\emptyset$ because $N_G(w)\setminus (N_G^2(u,w)\cup \{u,v\})=N_G^1(w)\cup \bar{N}_G^2(v,w)$. If $E_G(v, V(G'))=\emptyset$ then $G$ is a graph of order six, it can be easily checked that $G$ cannot be a counterexample. If $t_3=2$ then $S_0=N_G^2(u,w)\cup \{w,x\}$ induces an odd subgraph $K\cong K_{1,3}$, and therefore $H=K\cup H'$ is an odd induced subgraph of $G$ with $|V(H)|\ge 4+\frac 25|V(G')|>\frac 25|V(G)|$, a contradiction. Hence suppose $E_G(v, V(G'))\not=\emptyset$ and $t_3<2$. Reset $V_0=(N_G(u)\cup \{u,w, x\})\setminus\{v\}$ and $G'=G-V_0$. Then $G'$ has no isolated vertex and so, by the minimality of $G$, $G'$ has an odd induced subgraph $L'$ with $|V(L')|\ge \frac 25|V(G')|$. If $v\notin V(L')$ or $vw, vx\notin E(G)$, then $\{w,x\}\cup V(L')$ induces an odd subgraph $H$ of $G$ with $|V(H)|\ge 2+\frac 25|V(G')|\ge\frac 25|V(G)|$, a contradiction. So suppose $v\in V(L')$ and $vw\in E(G)$ or $vx\in E(G)$. If $N_G(v)\cap V_0$ has two nonadjacent vertices, say $\{a,b\}$, then $\{a,b\}\cup V(L')$ induces an odd subgraph of $G$ with  order at least $\frac 25|V(G)|$, a contradiction. This implies that $N_G^2(u,v)=\emptyset$ (i.e $t_2=0$), $vx\notin E(G)$ (i.e. $x\in N_G^1(w)$) and $vw, uw\in E(G)$ (otherwise, it is easy to choose two nonadjacent vertices from $N_G^2(u,v)\cup\{u,w,x\}$). As $t_1+t_2+t_3=2$, $t_2=0$, and $t_3<2$, we have $t_1>0$. Choose $a\in N_G^1(u)$, then  $\{a,u,w,x\}\cup V(L')$ induces an odd subgraph $H$ of $G$ with $|V(H)|\ge 4+\frac 25|V(G')|>\frac 25|V(G)|$, a contradiction.

\vspace{5pt}

\noindent{\bf Subcase 1.2}. $|N_G(w)\setminus (N_G^2(u,w)\cup \{u,v\})|\geq 2$.

Choose $x\in N_G^1(w)\cup \bar{N}_G^2(v,w)$ (this can be done because $N_G^1(w)\cup \bar{N}_G^2(v,w)\not=\emptyset$ by Claim~\ref{CLAIM: l6c2}) and set $V_0=(N_G(u)\cup \{u,v,x\})\setminus \{w\}$ and $G'=G-V_0$. Then $G'$ has no isolated vertex and so, by the minimality of $G$, $G'$ has an odd induced subgraph $H'$ with $|V(H')|\ge\frac 25|V(G')|$.

\begin{claim}\label{CLAIM: l6c3}
$w\in V(H')$.
\end{claim}
If $w\notin V(H')$, choose a maximum subset $S$ of $N_G(u)\setminus\{v, w\}$ so that $s=|S|$ is odd, then $S_0=S\cup\{u\}$ induces an odd subgraph $K\cong K_{1,s}$ of $G[V_0]$ such that $|S_0|/|V_0|=(s+1)/(t_1+t_2+t_3+3)\geq 2/5$. Clearly, $N_G(S_0)\cap V(H_0)=\emptyset$ and $S_0\cap N_G(V(H_0))=\emptyset$. Hence $H=K\cup H'$ is an odd induced subgraph of $G$ with $|V(H)|\ge \frac 25|V(G)|$, a contradiction. The claim holds.


\begin{claim}\label{CLAIM: l6c4}
 $t_3\le t_1+t_2$.
\end{claim}
If $t_3\ge t_1+t_2+1$, choose a maximum subset $S_0$ of $N_G^2(u,w)\cup \{x\}$ so that $|S_0|$ is even, then $|S_0|/|V_0|=|S_0|/(t_1+t_2+t_3+3)\geq |S_0|/(2t_3+2)\ge 2/5$ unless $t_3=2$ and $t_1+t_2=1$. Therefore, $S_0\cup V(H')$ induces an odd subgraph $H$ of $G$ with $|V(H)|\ge\frac 25|V(G)|$  unless $t_3=2$ and $t_1+t_2=1$.
For $t_3=2$ and $t_1+t_2=1$, reset $V_0=(N_G(u)\cup \{u,v\})\setminus \{w\}$ and $G'=G-V_0$,  then, again by the minimality of $G$,  $G'$ has an odd induced subgraph $L'$ with $|V(L')|\ge \frac 25|V(G')|$. If $w\in V(L')$, set $S_0=N_G^2(u,w)$, and if $w\notin V(L')$, set $S_0=\{u,y\}$, where $y$ is a vertex in $N_G^2(u,w)$. In both cases, $S_0\cup V(L')$ induces an odd subgraph $H$ of $G$ with $|V(H)|\ge \frac 25|V(G)|$. Therefore, we always obtain a contradiction and so the claim follows.

\vspace{5pt}

Now let $S$ be a maximum subset of $N_G^1(u)\cup N_G^2(u,v)$ so that $s=|S|$ is even if $uw\in E(G)$, and $s=|S|$ is odd if $uw\notin E(G)$. Set $S_0=S\cup \{u,x\}$ if $uw\in E(G)$ and $S_0=S\cup\{u\}$ if $uw\notin E(G)$. Clearly, $S_0\cup V(H')$ induces an odd subgraph $H$ of $G$ and furthermore, for $uw\in E(G)$, $|S_0|/|V_0|\ge (s+2)/(t_1+t_2+t_3+3)\ge (t_1+t_2+1)/(2t_1+2t_2+3)\geq 2/5$; and for $uw\notin E(G)$, $|S_0|/|V_0|=(s+1)/(t_1+t_2+t_3+3)\geq 2/5$ unless $t_1+t_2=2$, $t_3=1$ or $t_1+t_2=2$, $t_3=2$ or $t_1+t_2=t_3=4$. Therefore, but some exceptions, $H$ is an odd induced subgraph with $|V(H)|\ge\frac 25|V(G)|$, a contradiction.  Note that all the exceptions occur under the assumption $uw\notin E(G)$. In the following of the case, we show that each of the three exceptions cannot occur in the minimal counterexample $G$ as well.

For $t_1+t_2=2$ and $t_3=1$, reset $V_0=N_G(u)\cup \{u,v\}$ and let $G'=G-V_0$, then, by the minimality of $G$, $G'$ has an odd induced subgraph $L'$ with $|V(L')|\ge \frac 25|V(G')|$.  Choose a vertex $a\in N_G^1(u)\cup N_G^2(u,v)$, then $S_0=\{u,a\}$ induces an odd subgraph $K\cong K_{1,1}$ of $G[V_0]$. As $N_G(S_0)\cap V(L')=\emptyset$ and $S_0\cap N_G(V(L'))=\emptyset$, $H=K\cup L'$ is an odd induced subgraph of $G$ with $|V(H)|\ge \frac 25|V(G)|$, a contradiction.

For $t_1+t_2=t_3=2$. If $|N_G^1(w)\cup \bar{N}_G^2(v,w)|\le 2$, reset $V_0=N_G(u)\cup N_G^1(w)\cup \bar{N}_G^2(v,w)\cup\{u,v,w,x\}$, then $G'=G-V_0$ has no isolated vertex and so, by the minimality of $G$, $G'$ has an odd induced subgraph $L'$ with $|V(L')|\ge\frac 25|V(G')|$. Let $S$ be a subset of $N_G(u)\setminus\{v\}$ with $s=|S|=3$ (this can be done because $|N_G(u)|\ge t_1+t_2+t_3=4$). Then $S_0=S\cup \{u\}$ induces an odd subgraph $K\cong K_{1,3}$ of $G[V_0]$ and therefore $H=K\cup L'$ is an odd induced subgraph of $G$ with $|V(H)|\ge \frac 25|V(G)|$, a contradiction.
Now suppose $|N_G^1(w)\cup \bar{N}_G^2(v,w)|\ge 3$. Choose a vertex $y\in N_G^1(w)\cup \bar{N}_G^2(v,w)$ with $y\not=x$. Reset $V_0=N_G(u)\cup \{u,v,x,y\}$ and $G'=G-V_0$, then, by the minimality of $G$, $G'$ has an odd induced subgraph $L'$ with $|V(L')|\ge\frac 25|V(G')|$. Let $S_0=N_G^2(u,w)\cup \{x,y\}$  if $w\in V(L')$, and let $S_0=S\cup \{u\}$  if $w\notin V(L')$, where $S$ is a maximum subset of $N_G(u)\setminus\{v\}$ with $s=|S|=3$. Clearly, $S_0\cup V(L')$ induces an odd subgraph $H$ of $G$ with $|V(H)|>\frac 25|V(G)|$, a contradiction.

For $t_1+t_2=t_3=4$, reset $V_0=N_G(u)\cup \{u,v\}$ and $G'=G-V_0$, then $G'$ has no isolated vertices and so, by the minimality of $G$, $G'$ has an odd induced subgraph $L'$ with $|V(L')|\ge\frac 25|V(G')|$.   Let $S_0=N_G^2(u,w)$  if $w\in V(L')$, or let  $S_0=N_G^1(u)\cup N_G^2(u,v)\cup \{u\}$ if $w\notin V(L')$.
Then $|S_0|/|V_0|\ge 2/5$ and $S_0\cup V(L')$ induces an odd subgraph $H$ of $G$ with $|V(H)|\ge \frac 25|V(G)|$, a contradiction.

This proves Case 1. By symmetry, we may also assume $N_G^1(w)\neq \emptyset$ to verify the following remaining case.

\vspace{5pt}
\noindent{\bf Case 2}. $N_G^1(v)\neq \emptyset$.

 Choose $x\in N_G^1(v)$ and $y\in N_G^1(w)$, set $V_0=(N_G(u)\cup \{u,x,y\})\setminus \{v,w\}$ and $G'=G-V_0$.

 \begin{claim}\label{l6c5'}
 $G'$ has no isolated vertex.
 \end{claim}
Suppose to the contrary that $G'$ has isolated vertices. Then $v$ or $w$ must be an isolated vertex of $G'$. Without loss of generality, assume $v$ is an isolated vertex of $G'$. Then $D_G(v)=1$. But $N_G^1(v)\not=\emptyset$, this is a contradiction to Lemma~\ref{LEM: P1}.

\vspace{5pt}
Hence $G'$ has no isolated vertex and so, by the minimality of $G$, $G'$ has an odd induced subgraph $H'$ with $|V(H')|\ge\frac 25|V(G')|$.

\begin{claim}\label{CLAIM: l6c5}
$H'$ contains at least one of $\{v, w\}$.
\end{claim}
Suppose to the contrary that $H'$ contains none of $\{v, w\}$. Let $S$ be a maximum subset of $N_G(u)\setminus\{v,w\}$ so that $s=|S|$ is odd.
Then $S_0=S\cup \{u\}$ induces an odd subgraph $K\cong K_{1,s}$ with $|S_0|/|V_0|=(s+1)/(t_1+t_2+t_3+3)\geq 2/5$. Note that $N_G(S_0)\cap V(H')=\emptyset$ and $S_0\cap N_G(V(H'))=\emptyset$. Thus $H=K\cup H'$ is an odd induced subgraph of $G$ with $|V(H)|\ge \frac 25|V(G)|$, a contradiction.


\begin{claim}\label{CLAIM: l6c6}
If $w\in V(H')$ then $t_3\le t_1+t_2$. Symmetrically, if $v\in V(H')$ then $t_2\le t_1+t_3$.
\end{claim}
We show that $t_3\le t_1+t_2$ when $w\in V(H')$. Suppose to the contrary that $t_3\ge t_1+t_2+1$. Let $S_0$ be a maximum subset of $N_G^2(u,w)\cup \{y\}$ such that $|S_0|$ is even. Then $S_0\cup V(H')$ induces an odd subgraph $H$ of $G$ with $|V(H)|=|S_0|+|V(H')|\ge \frac 25|V_0|+\frac 25|V(G')|=\frac 25 |V(G)|$ unless $t_3=2$ and $t_1+t_2=1$.

For $t_3=2$ and $t_1+t_2=1$, reset $V_0=(N_G(u)\cup \{u,x\})\setminus \{v,w\}$ and $G'=G-V_0$, then $G'$ has no isolated vertex and so $G'$ has an odd induced subgraph $L'$ with $|V(L')|\ge \frac 25|V(G')|$ by the minimality of $G$. If $w\in V(L')$, let $S_0= N_G^2(u,w)$, then  $S_0\cup V(L')$ induces an odd subgraph $H$ with $|V(H)|\ge \frac 25|V(G)|$. Now suppose $w\notin V(L')$. If $v\notin V(L')$, choose a vertex $z$ from $N^1_G(u)\cup N^2_G(u,v)$, then $\{u, z\}\cup V(L')$ induces an odd subgraph $H$ of $G$ with $|V(H)|\ge \frac 25|V(G)|$. Hence $v\in V(L')$, choose a vertex  $z\in N^2_G(u,v)\cup\{u\}$ which is adjacent to $v$, then $\{x,z\}\cup V(L')$ induces an odd subgraph $H$ of $G$ with $|V(H)|\ge \frac 25|V(G)|$.
In all cases we get contradictions and so the claim follows.

\vspace{5pt}
{ We shall show that certain special cases cannot occur in the minimal counterexample $G$, which would be helpful to eliminate exception values in later discussion.

\begin{claim}\label{CLAIM: l6c9}
If $uw\notin E(G)$ then none of the following occurs in the minimal counterexample $G$.\\
(a) $t_1+t_2=2$ and $t_3=1$;\\
(b) $t_1+t_2=t_3=p, \ p=2$ or $4$.
\end{claim}
}

For $t_1+t_2=2$ and $t_3=1$,  reset $V_0=(N_G(u)\cup\{u,x\})\setminus\{w, v\}$ and $G'=G-V_0$, then $G'$ has no isolated vertex and so, by the minimality of $G$, $G'$ has an odd induced subgraph $L'$ with $|V(L')|\ge \frac 25|V(G')|$. If $v\in V(L')$, choose a vertex  $z\in N^2_G(u,v)\cup\{u\}$ which is adjacent to $v$, note that $uw\notin E(G)$, then $\{x,z\}\cup V(L')$ induces an odd subgraph $H$ of $G$ with $|V(H)|\ge \frac 25|V(G)|$, a contradiction. Hence $v\notin V(L')$, choose a vertex $z$ from $N^1_G(u)\cup N^2_G(u,v)$, then $\{u, z\}\cup V(L')$ induces an odd subgraph $H$ of $G$ with $|V(H)|\ge \frac 25|V(G)|$, a contradiction.

For $t_1+t_2=t_3=p$, $p=2,4$, reset $V_0=(N_G(u)\cup\{u\})\setminus\{v,w\}$ and $G'=G-V_0$, then $G'$ has no isolated vertex and so, by the minimality of $G$, $G'$ has an odd induced subgraph $L'$ with $|V(L')|\ge\frac 25|V(G')|$. If $w\in V(L')$, note that $|N^2_G(u,w)|=t_3=p$ is even, $N^2_G(u,w)\cup V(L')$ induces an odd subgraph $H$ of $G$ with $|V(H)|\ge \frac 25|V(G)|$, a contradiction. So suppose $w\notin V(L')$.
If $uv\notin E(G)$ or $v\notin V(L')$, choose a subset $S$ of $N^2_G(u,w)$ so that $|S|=p-1$, then $S\cup\{u\}\cup V(L')$ induces an odd subgraph $H$ of $G$ with $|V(H)|\ge\frac 25|V(G)|$, a contradiction. Hence $uv\in E(G)$ and $v\in V(L')$. If $x\in V(L')$, then $N^2_G(u,w)\cup\{u\}\cup (V(L')\setminus\{x\})$ induces an odd subgraph $H$ with $|V(H)|=p+1+|V(L')|-1\ge \frac 25 |V(G)|$, a contradiction. Hence $x\notin V(L')$. Then $N^2_G(u,w)\cup\{u\}\cup V(L')\cup\{x\}$ induces an odd subgraph $H$ of $G$ with $|V(H)|\ge \frac 25|V(G)|$, a contradiction. This proves the claim.

\vspace{5pt}
By Claim~\ref{CLAIM: l6c5}, we may assume, without loss of generality,  $w\in V(H')$. Hence, by Claim~\ref{CLAIM: l6c6}, $t_3\le t_1+t_2$. Now we divide the discussion into two subcases below.

\vspace{5pt}

\noindent{\bf Subcase 2.1}.  $v\notin V(H')$.


\vspace{5pt}
 Let $S$ be a maximum subset of $N_G^1(u)\cup N_G^2(u,v)$ such that $s=|S|$ is odd if $uw\notin E(G)$ or $s=|S|$ is even if $uw\in E(G)$. Set $S_0=S\cup\{u\}$ if $uw\notin E(G)$  or $S_0=S\cup\{u,y\}$ if $uw\in E(G)$. Note that $s=t_1+t_2$ or $t_1+t_2-1$ depending on the parity of $t_1+t_2$ and $|S_0|=s+1$ or $s+2$ depending on $uw\notin E(G)$ or $uw\in E(G)$. Notice that $t_3\le t_1+t_2$, we have $|S_0|/|V_0|=|S_0|/(t_1+t_2+t_3+3)\geq 2/5$ unless $uw\notin E(G)$ and $t_1+t_2=2$, $t_3=1$, or $t_1+t_2=t_3=2$, or $t_1+t_2=t_3=4$. Therefore $S_0\cup V(H')$ induces an odd subgraph $H$ of $G$ with $|V(H)|\ge \frac 25|V(G)|$ but three exceptions. However, none of the exceptions occur in $G$ by Claim \ref{CLAIM: l6c9}. This yields a contradiction and verifies Subcase 2.1.

 \vspace{5pt}

{ \noindent{\bf Subcase 2.2}.  $v\in V(H')$.

By Claim~\ref{CLAIM: l6c6}, we have $t_3\le t_1+t_2$ and $t_2\le t_1+t_3$. Furthermore, we have the following claim.
\begin{claim}\label{CLAIM: l6c8}
We have $t_2+t_3\leq t_1$.
\end{claim}
Suppose to the contrary that $t_2+t_3\geq t_1+1$.
Let $S_v$ be a maximum subset of $ N_G^2(u,v)\cup\{x\}$ such that $|S_v|$ is even, let $S_w$ be a maximum subset of $ N_G^2(u,w)\cup\{y\}$ such that $|S_w|$ is even, and set $S_0=S_u\cup S_v$. Then $S_0\cup V(H')$ induces an odd subgraph of $G$. By checking the parity of $t_2$ and $t_3$ with certain calculation,
we have $|S_0|/|V_0|\ge \frac{2}{5}$ unless $t_1=1$, $t_2+t_3=2$ and $t_i$, $i=2,3$, is even. But this exception cannot occur because $t_3\le t_1+t_2$ and $t_2\le t_1+t_3$, a contradiction. Hence the claim holds.

Now, we choose a set $S_0$ according to the following rules:

{\em (i) If $uv\in E(G), uw\in E(G)$, let $S_0=S_u\cup \{u, x, y\}$, where $S_u$ is the maximum subset of $N_G^1(u)$ with  size odd;

     (ii) If $uv\in E(G), uw\notin E(G)$, let $S_0=S_u\cup \{u, x\}$, where $S_u$ is the maximum subset of $N_G^1(u)$ with  size even;

     (iii) If $uv\notin E(G), uw\in E(G)$, let $S_0=S_u\cup \{u, y\}$, where $S_u$ is the maximum subset of $N_G^1(u)$ with  size even;

     (iv) If $uv\notin E(G), uw\notin E(G)$, let $S_0=S_u\cup \{u\}$, where $S_u$ is the maximum subset of $N_G^1(u)$ with  size odd.
}

Then $S_0\cup V(H')$ induces an odd subgraph of $G$ by definition. It remains to compute $|S_0|/|V_0|$.

If $t_1$ is odd, we have $|S_0|/|V_0|\ge(t_1+1)/(t_1+t_2+t_3+3)\geq 2/5$ by Claim \ref{CLAIM: l6c8} in each of the cases (i)-(iv). If $t_1$ is even, it follows from Claim \ref{CLAIM: l6c8} that $|S_0|/|V_0|\ge(t_1+2)/(t_1+t_2+t_3+3)\geq 2/5$ in each of the cases (i)-(iii), and in the case (iv), $|S_0|/|V_0|=t_1/(t_1+t_2+t_3+3)\geq 2/5$ unless $t_1=2$, $t_2+t_3=2$ or $t_1=4$, $t_2+t_3=4$.
 Therefore, $S_0\cup V(H')$ induces an odd subgraph $H$ of $G$ with $|V(H)|\ge \frac 25|V(G)|$ unless $t_1=t_2+t_3=2$ or $t_1=t_2+t_3=4$.

For $t_1=t_2+t_3=p$, $p=2,4$, reset $V_0=N_G(u)\cup\{u\}$ and $G'=G-V_0$, then $G'$ has no isolated vertex and so, by the minimality of $G$, $G'$ has an odd induced subgraph $L'$ with $|V(L')|\ge\frac 25|V(G')|$. Choose a subset $S$ of $N^1_G(u)$ so that $|S|=p-1$, then $S\cup\{u\}\cup V(L')$ induces an odd subgraph $H$ of $G$ with $|V(H)|\ge\frac 25|V(G)|$, a contradiction.


}

The proof of the lemma is completed.
\end{proof}

The following three structural properties of the minimum counterexample $G$ are direct consequence of Lemmas~\ref{LEM: P1} and~\ref{LEM: P2}.

\begin{cor}\label{COR: c1}
Let $V_1$ be the set of all 1-vertices in $G$ and let $P=N_G(V_1)$. Suppose $G_1=G-V_1$, then $d_{G_1}(x)\geq 3$ for any $x\in P$.
\end{cor}

\begin{proof}[Proof]
Suppose to the contrary that there is a vertex $x\in P$ with $d_{G_1}(x)\le 2$. If $d_{G_1}(x)=0$ then $G$ is isomorphic to a star, which cannot be a counterexample. Hence $0<d_{G_1}(x)\le 2$. This implies that $0<D_G(x)\le 2$. But $|N_G^1(x)|\geq 1$, this is a contradiction to Lemmas~\ref{LEM: P1} or~\ref{LEM: P2}.
\end{proof}

\begin{cor}\label{COR: c2}
$G$ has no adjacent $2$-vertices.
\end{cor}

\begin{proof}[Proof]
Suppose to the contrary that $G$ has two adjacent $2$-vertices $u,v$. Then $D_G(u)\le 2$. Let $v_1=N_G(v)\setminus \{u\}$.  Then $v\in N^2_G(u, v_1)$, which is a contradiction to Lemmas~\ref{LEM: P1} or~\ref{LEM: P2}.


\end{proof}

\begin{cor}\label{COR: c3}
$G$ has no vertex $u$ with $d_G(u)\ge 3$ so that $D_G(u)\le 2$.
\end{cor}
\begin{proof}[Proof]
Suppose to the contrary that $G$ has  a vertex $u$ with $d_G(u)\ge 3$ and $D_G(u)\le 2$. By Lemmas~\ref{LEM: P1} and~\ref{LEM: P2}, $u$ has no neighbor of degree at most 2 since $G$ cannot be isomorphic to a star. This implies $D_G(u)\ge d_G(u)\ge 3$, a contradiction.
\end{proof}

\section{Proof of Theorem~\ref{THM: Main}}
Before giving the proof, we need some definition and structural properties of graphs with treewidth at most 2. A graph $G$ contains a graph $H$ as a {\it minor} if $H$ can be obtained from a subgraph of $G$ by contracting edges, and $G$ is called {\it $H$-minor} free if $G$ does not have $H$ as a minor.
It is well known that
\begin{prop}\label{PROP: p1}[Proposition 12.4.2,~\cite{Diestel-GT00}]
A graph has treewidth at most 2 if and only if it is $K_4$-minor free.
\end{prop}

For $K_4$-minor free graphs, Lih, Wang, and Zhu~(\cite{LWZ-DM03}) gave a powerful structural property of them.

\begin{lem}\label{LEM: K_4-free}[Lemma 2,~\cite{LWZ-DM03}]
 If $G$ is a $K_4$-minor free graph, then one of the following holds:

 (a) $\delta(G)\leq 1$;

 (b) there exist two adjacent $2$-vertices;

 (c) there exists a vertex $u$ with $d_G(u)\geq 3$ such that $D_G(u)\leq 2$.

\end{lem}

\begin{proof}[Proof of Theorem~\ref{THM: Main}]
Let $G$ be a minimum counterexample with respect to the order of $G$. By the minimality of $G$, $G$ must be connected.
Let $V_1$ be the set of all 1-vertices in $G$ and $P=N_G(V_1)$. Let $G_1=G-V_1$. By Corollaries~\ref{COR: c1} and~\ref{COR: c2}, $\delta(G_1)\geq 2$ and $G_1$ has no adjacent $2$-vertices. Clearly, $tw(G_1)\le 2$ and hence $G_1$ is $K_4$-minor free. By Lemma~\ref{LEM: K_4-free},  $G_1$ has a vertex $u$ with $d_{G_1}(u)\geq 3$ and  $D_{G_1}(u)\leq 2$. Clearly, $d_G(u)=d_{G_1}(u)+|N^1_G(u)|$ and the adding of the vertices of $N^1_G(u)$ to $G_1$ does not increases the value of $D_{G_1}(u)$. So $D_G(u)=D_{G_1}(u)\leq 2$, this is a contradiction to Corollary~\ref{COR: c3}.
The proof of Theorem~\ref{THM: Main} is completed.
\end{proof}

\section{Concluding remarks}
Let $$\mathcal{G}_k=\{G\colon\, tw(G)\le k \mbox{ and $G$ contains no isolated vertex}\},$$
and $c_k=\min_{G\in \mathcal{G}_k}\frac{f(G)}{|V(G)|}$. { Since each graph in $\mathcal{G}_k$ has chromatic number at most $k+1$, Scott's result~\cite{Scott-CPC92} implies $c_k\ge \frac{1}{2k+2}$. The follow graphs $H_k$ in Figure~\ref{FIG: Hk} gives an upper bound $c_k\le \frac{2}{k+3}$ for $k=1,2,3,4$. Note that the graph $H_4$ is found by Caro~\cite{Caro-DM1994}, which is the smallest known ratio of $\frac{f(G)}{|V(G)|}$ for all graphs $G$.}
As we have known, {Theorem~\ref{THM: RSMain} of Radcliffe and Scott~\cite{RS-DM95} and the upper bound of $c_k$ implies $c_1=1/2$}, and in this paper, we show that $c_2=2/5$ (Theorem~\ref{THM: Main}).
As a far more step, we want ask the question: what is the exact value $c_k$ for  graphs in $\mathcal{G}_k$. { It is plausible that $c_3=\frac{1}{3}$ and $c_4=\frac{2}{7}$.}

  \begin{figure}[ht]

\setlength{\unitlength}{0.08cm}

\begin{center}

\begin{picture}(110,35)

\put(-3, 5){\circle*{1.6}}\put(9, 5){\circle*{1.6}}\put(-3, 15){\circle*{1.6}}\put(9, 15){\circle*{1.6}}
\put(0, -4){\small $H_1$}
\qbezier(-3, 5)(-3, 5)(9, 5)\qbezier(-3, 5)(-3, 5)(-3, 15)\qbezier(9, 5)(9, 15)(9, 15)

\put(46.4127, 19.63525){\circle*{1.6}}\put(23.5873, 19.63525){\circle*{1.6}}
\put(35, 27){\circle*{1.6}}\put(42.0534, 5.2918){\circle*{1.6}}\put(27.9466, 5.2918){\circle*{1.6}}
\put(32, -4){\small $H_2$}

\qbezier(35, 27)(46.4127, 19.63525)(46.4127, 19.63525)
\qbezier(42.0534, 5.2918)(42.0534, 5.2918)(46.4127, 19.63525)
\qbezier(27.9466, 5.2918)(27.9466, 5.2918)(42.0534, 5.2918)
\qbezier(23.5873, 19.63525)(23.5873, 19.63525)(27.9466, 5.2918)
\qbezier(35,27)(35,27)(23.5873, 19.63525)

\put(63.5, 3.75){\circle*{1.6}}
\put(76.5, 3.75){\circle*{1.6}}
\put(83, 15){\circle*{1.6}}
\put(76.5, 26.25){\circle*{1.6}}
\put(63.5, 26.25){\circle*{1.6}}
\put(57, 15){\circle*{1.6}}

\put(67, -4){\small $H_3$}

\qbezier(63.5, 3.75)(63.5, 3.75)(76.5, 3.75)
\qbezier(63.5, 3.75)(63.5, 3.75)(57, 15)
\qbezier(57, 15)(57, 15)(63.5, 26.25)
\qbezier(63.5, 26.25)(63.5, 26.25)(76.5, 26.25)
\qbezier(76.5, 26.25)(76.5, 26.25)(83, 15)
\qbezier(83, 15)(83, 15)(76.5, 3.75)

\qbezier(76.5, 26.25)(76.5, 26.25)(76.5, 3.75)
\qbezier(57, 15)(57, 15)(76.5, 3.75)

\qbezier(57, 15)(57, 15)(76.5, 26.25)
\qbezier(83, 15)(83, 15)(63.5, 3.75)
\qbezier(83, 15)(83, 15)(63.5, 26.25)

\put(114.3820, 22.4819){\circle*{1.6}}
\put(116.6991, 12.32975){\circle*{1.6}}
\put(110.2066, 4.1884){\circle*{1.6}}
\put(99.7934, 4.1884){\circle*{1.6}}
\put(93.3009, 12.32975){\circle*{1.6}}
\put(95.6180, 22.4819){\circle*{1.6}}
\put(105, 27){\circle*{1.6}}
\put(102, -4){\small $H_4$}

\qbezier(114.3820, 22.4819)(114.3820, 22.4819)(116.6991, 12.32975)
\qbezier(114.3820, 22.4819)(114.3820, 22.4819)(110.2066, 4.1884)
\qbezier(114.3820, 22.4819)(114.3820, 22.4819)(95.6180, 22.4819)

\qbezier(110.2066, 4.1884)(110.2066, 4.1884)(116.6991, 12.32975)
\qbezier(93.3009, 12.32975)(110.2066, 4.1884)(110.2066, 4.1884)
\qbezier(105, 27)(116.6991, 12.32975)(116.6991, 12.32975)
\qbezier(105, 27)(93.3009, 12.32975)(93.3009, 12.32975)
\qbezier(95.6180, 22.4819)(99.7934, 4.1884)(99.7934, 4.1884)
\qbezier(116.6991, 12.32975)(99.7934, 4.1884)(99.7934, 4.1884)

\qbezier(110.2066, 4.1884)(110.2066, 4.1884)(99.7934, 4.1884)

\qbezier(99.7934, 4.1884)(93.3009, 12.32975)(93.3009, 12.32975)

\qbezier(95.6180, 22.4819)(93.3009, 12.32975)(93.3009, 12.32975)

\qbezier(95.6180, 22.4819)(95.6180, 22.4819)(105, 27)

\qbezier(105, 27)(105, 27)(114.3820, 22.4819)

\end{picture}
\end{center}
\caption{\small\it Graphs  $H_k$ with treewidth $k$ and $\frac{f(H_k)}{|V(H_k)|}=\frac{2}{k+3}$ for $k=1,2,3,4$.}
\label{FIG: Hk}
\end{figure}
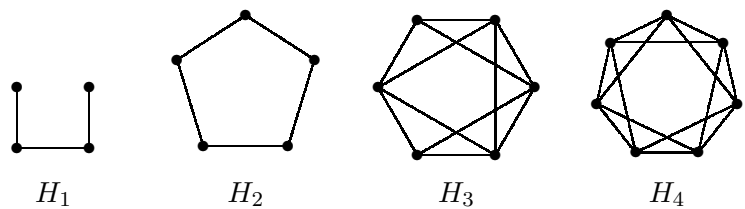



\begin{thebibliography}{99}
\bibitem{BWW-AJC97}
D. M. Berman, H. Wang, L. Wargo, Odd induced subgraphs in graphs of maximum degree three. Aust. J. Comb. (1997) 81-85.






\bibitem{Caro-DM1994}
Y. Caro, On induced subgraphs with odd degrees. Discrete Math., (1994) 23-28.

\bibitem{Diestel-GT00}
R. Diestel, Graph Theory, Springer-Verlag New York, 2000.

\bibitem{LWZ-DM03}
K. W. Lih, W. F. Wang, X. D. Zhu, Coloring the square of a $K_4$-minor free graph. Discrete Math. (2003) 303-309.

\bibitem{Lovasz-CPE1979}
L. Lov\'{a}sz, Combinatorial Problems and Exercises. (North-Holland, Amsterdam, 1979).

\bibitem{RS-DM95}
A. J. Radcliffe, A. D. Scott, Every tree contains a large induced subgraph with all degrees odd. Discrete Math. (1995) 275-279.






\bibitem{Scott-CPC92}
A. D. Scott, Large induced subgraphs with all degrees odd. Comb. Probab. Comput. 1(1992) 335-349.


\end{thebibliography}
\end{document}